\documentclass[10pt, a4paper]{article} 

\usepackage[utf8]{inputenc} 

\usepackage{graphicx} 
\usepackage[english]{babel}
\usepackage{booktabs} 
\usepackage{array} 
\usepackage{caption}

\usepackage{amsfonts,amsmath,amssymb,bbm,amsthm,amsbsy}
\usepackage[all]{xy}
\SelectTips{cm}{}

\usepackage[pagebackref,
    ,pdfborder={0 0 0}
  ,urlcolor=black,a4paper,hypertexnames=false]{hyperref}

\usepackage{color}
\usepackage{pdfcolmk}

\makeatletter
\def\blfootnote{\xdef\@thefnmark{}\@footnotetext}
\makeatother
\date{\today%
    \protect\blfootnote{\copyright{\ N.~Heuer, C.~L\"oh 2019}. 
    This work was supported by the CRC~1085 \emph{Higher Invariants} 
    (Universit\"at Regensburg, funded by the DFG).
    \\
    MSC~2010 classification: 57N65, 57M07, 20J05, 11J86, 03D78}}

\usepackage[nouppercase]{scrlayer-scrpage}

\automark[subsection]{section}
\usepackage{ifthen}
\deftripstyle{myheadings}
             {\ifthenelse{\isodd{\value{page}}}{\leftmark}{\pagemark}}
             {}
             {\ifthenelse{\isodd{\value{page}}}{\pagemark}{\rightmark}}%
             {}{}{}%
\def\args{\;\cdot\;}

\def\varepsilon{\epsilon}
\def\longrightarrow{\to}
\def\longmapsto{\mapsto}
\def\phi{\varphi}

\newtheorem{thm}{Theorem}[section]
\newtheorem{lemma}[thm]{Lemma}
\newtheorem{prop}[thm]{Proposition}
\newtheorem{corr}[thm]{Corollary}
\newtheorem{claim}[thm]{Claim}

\theoremstyle{remark}
\newtheorem{rmk}[thm]{Remark}

\theoremstyle{definition}
\newtheorem{defn}[thm]{Definition}
\newtheorem{exmp}[thm]{Example}
\newtheorem{setup}[thm]{Setup}
\newtheorem{quest}[thm]{Question}

\theoremstyle{theorem}
\newtheorem{theorem}{Theorem}

\usepackage{tikz}
\usepackage{tikz-cd}

\usepackage{IEEEtrantools}

\newenvironment{equ*}[1]{\begin{IEEEeqnarray*}{#1}}{\end{IEEEeqnarray*}}

\newcommand{\R}{\mathbb{R}}

\newcommand{\Q}{\mathbb{Q}}
\newcommand{\Z}{\mathbb{Z}}
\newcommand{\N}{\mathbb{N}}
\DeclareMathOperator{\scl}{scl}
\newcommand{\Const}{\mathbb{R}^c_{>0}}
\newcommand{\Prim}{\mathbb{P}}

\newcommand{\eurm}{\mathrm{eu}}

\def\eu#1#2{%
  {}^{#2}\eurm^{#1}}

\DeclareMathOperator{\rot}{rot}

\DeclareMathOperator{\SV}{SV}
\DeclareMathOperator{\id}{id}
\DeclareMathOperator{\im}{im}
\DeclareMathOperator{\Hom}{Hom}

\DeclareMathOperator{\trig}{trig}

\newcommand{\col}{\colon}

\def\qand{\quad\text{and}\quad}

\def\exi#1{%
  \exists_{#1}\;\;\;}

\DeclareMathOperator{\SL}{SL}
\def\SLS{%
  \SL_2(\Z[1/2])}
\DeclareMathOperator{\tr}{tr}
\DeclareMathOperator{\Ext}{Ext}

\title{Transcendental simplicial volumes}
\author{Nicolaus Heuer, Clara L\"oh}

\begin{document}

\maketitle

\begin{abstract}
  We show that there exist closed manifolds with arbitrarily small
  transcendental simplicial volumes. Moreover, we exhibit an explicit
  family of (transcendental) real numbers that are
  \emph{not} realised as the simplicial volume of a closed manifold.
\end{abstract}

\section{Introduction}

The simplicial volume~$\|M\| \in \R_{\geq 0}$ is a homotopy invariant
of oriented closed connected manifolds~$M$~\cite{munkholm,vbc}, namely
the $\ell^1$-semi-norm of the (singular) $\R$-fundamental class.  The
set~$\SV(d) \subset \R_{\geq 0}$ of simplicial volumes of oriented
closed connected $d$-manifolds is countable and can be determined
explicitly in dimensions~$1$,~$2$,~$3$ through classification
results~\cite[Section 2.2]{heuerloeh4mfd}. In these dimensions, simplicial
volume has a gap at~$0$.

In previous work~\cite{heuerloeh4mfd}, we showed that those are the
only dimensions with a gap and that indeed $\SV(d)$ is
dense in~$\R_{\geq 0}$ for $d \in \N_{\geq 4}$.  We also showed that
$\SV(4)$ contains~$\Q_{\geq 0}$.  We now continue these
investigations, with a focus on transcendental values.

\begin{theorem}\label{theorem:transsimvol} 
  For every~$\varepsilon \in \R_{>0}$, there exists an oriented closed
  connected $4$-manifold~$M$ such that
  \begin{itemize}
  \item $\|M\|$ is transcendental (over~$\Q$) and
  \item $0 < \| M \| < \varepsilon$.
  \end{itemize}
\end{theorem}

In fact, we provide an explicit sequence of transcendental simplicial
volumes of $4$-mani\-folds converging to zero that are linearly
independent over the algebraic numbers
(Theorem~\ref{thm:transsimvol_explicit}). 

We also give explicit examples of real numbers that are not realised
as a simplicial volume: 

\begin{theorem}\label{theorem:nonsimvol}
  Let~$d\in \N$ and let $A \subset \N$ be a subset that is
  recursively enumerable but not recursive. Then
  \[ \alpha := \sum_{n \in A} 2^{-n}
  \]
  is transcendental (over~$\Q$) and there is \emph{no} oriented closed
  connected $d$-mani\-fold~$M$ with~$\|M\| \in \Const \cdot \alpha$,
  where $\Const$ is the set of positive computable
  numbers.
\end{theorem}

There are many recursively enumerable but non-recursive subsets of
$\N$: for example, every encoding of the halting sequence~\cite[Section
  7]{Cutland}; moreover, $1 \in \Const$. Hence,
Theorem~\ref{theorem:nonsimvol} provides concrete examples of
countably many transcendental numbers that are \emph{not} realised as
the simplicial volume of closed manifolds.

We previously explored connections between stable commutator length on
finitely presented groups and simplicial
volume~\cite{heuerloeh_onerel}\cite[Theorem~C/F]{heuerloeh4mfd}; see
also Theorem \ref{thm:simvolscl}. Stable commutator length is now well studied in many classes of groups, thanks largely to Calegari and others \cite{Calegari, Calegari_rational, zhuang, CalegariFujiwara, ChenHeuer}. Our constructions for the transcendental values of simplicial volumes in Theorems \ref{theorem:transsimvol} and \ref{thm:transsimvol_explicit}
rely on computations by Calegari~\cite[Chapter~5]{Calegari}.

However, it is unknown which real non-negative numbers
are generally realised as the stable commutator length of elements in finitely
presented groups. For the larger class of \emph{recursively} presented
groups, the set of stable commutator length is known and coincides
with the set of right-computable
numbers~\cite{Heuer-scl-rp-groups}. Thus we ask:

\begin{quest}
  Does the set of simplicial volumes of oriented closed connected
  $4$-man\-i\-folds coincide with the set of non-negative right-computable real
  numbers?
\end{quest}

\subsection*{Proof of Theorem~\ref{theorem:transsimvol}}

Theorem~\ref{theorem:transsimvol} will follow from the following
explicit construction of simplicial volumes:

\begin{theorem} \label{thm:transsimvol_explicit}
  There exists a constant~$K \in \N_{>0}$ and a sequence $(M_n)_{n \in \N}$ of
  oriented closed connected $4$-mani\-folds with
  \[
  \| M_n \| = K \cdot \frac{24 \cdot \arccos(1 - 2^{-n-1})}{\pi}
  \]
  for all~$n\in \N$. 
  The numbers~$\alpha_n := 24 \cdot \arccos(1-2^{-n-1})/\pi$ have the
  following properties:
  \begin{enumerate}
  \item We have~$\lim_{n \rightarrow \infty} \alpha_n = 0$.
  \item We have~$\alpha_0 = 8$ and for each~$n \in \N_{>0}$, the number~$\alpha_n$ is
    transcendental~(over~$\Q$).
  \item The family~$(\alpha_{p-2})_{p \in \Prim}$ is linearly independent
    over the field of algebraic numbers; here, $\Prim \subset \N$
    denotes the set of all prime numbers.
  \end{enumerate}
\end{theorem}

The simplicial volumes constructed in
Theorem~\ref{thm:transsimvol_explicit} will be based on our previous
work~\cite{heuerloeh4mfd} that allows us to construct $4$-manifolds
with simplicial volumes prescribed in terms of the stable commutator
length of certain finitely presented groups. See Calegari's
book~\cite{Calegari} for background on stable commutator length.

\begin{thm}[\protect{\cite[Theorem~F]{heuerloeh4mfd}}]\label{thm:simvolscl}
  Let $\Gamma$ be a finitely presented group that
  satisfies~$H_2(\Gamma;\R) \cong 0$ and let $g \in [\Gamma,\Gamma]$
  be an element in the commutator subgroup. Then there exists an oriented
  closed connected $4$-manifold~$M_g$ with
  \[ \| M_g \| = 48 \cdot \scl_\Gamma g.
  \]
\end{thm}

As input for this theorem, we use the following group (whose properties 
are established in Section~\ref{sec:SLSE}):

\begin{theorem}\label{theorem:SLSEnew}
  The central extension~$\widetilde \Gamma$ of~$\SLS$ corresponding
  to the integral Euler class of~$\SLS$ is finitely presented. Moreover,
  $H_1(\widetilde \Gamma;\Z)$ is finite and~$H_2(\widetilde \Gamma ;\R) \cong 0$.
\end{theorem}

It is known that the image of stable commutator length of
the central Euler class extension of~$\SLS$ contains arbitrarily small
transcendental numbers~\cite[Example~5.38]{Calegari}: 

\begin{exmp}\label{exa:sclrot}
  Let $\Gamma := \SLS$ and let $\widetilde \Gamma$ denote the central
  extension of~$\SLS$ corresponding to the integral Euler class
  of~$\SLS$. In other words, $\widetilde \Gamma$ is the pre-image
  of~$\SLS$ under the canonical projection~$\widetilde\SL_2(\R)
  \longrightarrow \SL_2(\R)$, where $\widetilde\SL_2(\R)$ denotes the
  universal covering group of~$\SL_2(\R)$.
  Then
  \[ \scl_{\widetilde \Gamma} (\widetilde g) = \frac{\lvert\rot( \widetilde g)|}2
  \]
  for all~$\widetilde g \in \widetilde \Gamma$,
  where $\rot \col \widetilde \Gamma \to \R_{\geq 0}$
  denotes the rotation number~\cite[Example~5.38]{Calegari}. 
  
  Furthermore, for each~$g \in \Gamma$ with~$\lvert\tr(g)| \leq 2$,
  there is a lift $\widetilde{g} \in \widetilde{\Gamma}$ of $g$ such
  that~\cite[p.~145]{Calegari} 
  \[ \rot(\widetilde g) = \frac{\arccos (\tr g /2)}{\pi}.
  \]
  For~$n \in \N_{>0}$, we consider 
  \[ g_n :=
     \begin{pmatrix}
        2 & 1 + 2^{-n+1} \\
       -1 & -2^{-n}
     \end{pmatrix}
     \in \Gamma
  \]
  and let $\widetilde g_n \in \widetilde \Gamma$ be the associated lift. 
  Then $\lim_{n \rightarrow \infty} \rot(\widetilde g_n) = 0$ and 
  \[   \scl_{\widetilde \Gamma}(\widetilde g_n)
  = \frac{\lvert\rot(\widetilde g_n)|}2
  = \frac{\arccos(\tr g_n / 2)}{2 \cdot \pi}
  = \frac{\arccos(1 - 2^{-n-1})}{2 \cdot \pi}
  = \frac{\alpha_n}{48}.
  \]
  However, a priori, it is not clear that $\widetilde g_n$ lies in the
  commutator subgroup of~$\widetilde \Gamma$.  Because $K :=
  |H_1(\widetilde \Gamma;\Z)|$ is finite
  (Theorem~\ref{theorem:SLSEnew}), we know that $h_n := \widetilde g_n{}^K
  \in [\widetilde \Gamma,\widetilde \Gamma]$ for all~$n \in \N$. Moreover,
  by construction, 
  \[ \scl_{\widetilde \Gamma} (h_n)
  = K \cdot \scl_{\widetilde \Gamma} (\widetilde g_n)
  = K \cdot \frac{\alpha_n}{48}.
  \]
\end{exmp}

With these ingredients, we can complete the proof of
Theorem~\ref{thm:transsimvol_explicit} (and thus of
Theorem~\ref{theorem:transsimvol}):

\begin{proof}[Proof of Theorem~\ref{thm:transsimvol_explicit}/\ref{theorem:transsimvol}] 
Let $\widetilde \Gamma$ be the
central Euler class extension of~$\SLS$ and let $(h_n)_{n\in \N}$ and~$K$ be as in
Example~\ref{exa:sclrot}.  Applying
Theorem~\ref{thm:simvolscl} to~$h_n \in [\widetilde \Gamma,
  \widetilde \Gamma]$ results in an oriented closed connected
$4$-manifold~$M_n$ with~$\|M_n\| = K \cdot \alpha_n$. Hence, $\lim_{n
  \rightarrow \infty} \|M_n\| = K \cdot 24 \cdot \arccos(1) /\pi = 0$.  If
$n>0$, then $\alpha_n$ is known to be transcendental
(Proposition~\ref{prop:antrans}). Moreover, Baker's theorem proves the
last part of Theorem~\ref{thm:transsimvol_explicit}
(Proposition~\ref{prop:anindep}).
\end{proof}

\subsection*{Proof of Theorem~\ref{theorem:nonsimvol}}

The proof of Theorem~\ref{theorem:nonsimvol} relies on the
following simple observation (proved in Section~\ref{sec:simvolrc},
where also the definition of right-computability is recalled):

\begin{theorem}\label{theorem:simvolrc}
  Let $M$ be an oriented closed connected manifold. Then $\| M \|$
  is a right-computable real number.
\end{theorem}

In contrast, the numbers~$\alpha$ in
Theorem~\ref{theorem:nonsimvol} are \emph{not}
right-computable~(Proposition \ref{prop:not right computable}) and thus, in particular, \emph{not}
algebraic, because every algebraic number is
computable~\cite[Section~6]{eisermann}. The product of a
computable number with a number that is not right-computable is also
not right-computable (Section~\ref{subsec:rc}). Therefore,
applying Theorem~\ref{theorem:simvolrc} proves
Theorem~\ref{theorem:nonsimvol}.

\subsection*{Organisation of this article}

In Section~\ref{sec:trans}, we prove the transcendence properties of
the $\arccos$-terms.  In Section~\ref{sec:SLSE}, we solve the
group-theoretic problem for the proof of
Theorem~\ref{theorem:SLSEnew}. In Section~\ref{sec:simvolrc}, we prove
Theorem~\ref{theorem:simvolrc}.

\subsection*{Acknowledgements}

We would like to thank the anonymous referee for asking questions
about the Euler extension (which lead to
a simplification of the treatment of Theorem~\ref{theorem:SLSEnew}). 

\section{Some transcendental numbers}\label{sec:trans}

In this section, for~$n \in \N_{\geq 0}$, we will investigate the
transcendence of the following real numbers
\[ \alpha_{n} := \frac{24 \cdot \arccos(1 - 2^{-n-1})}{\pi}.
\] 
We will see that $\alpha_0 = 8$ and that $\alpha_n$ is transcendental (over the algebraic numbers) for every $n \geq 1$.  
\subsection{Transcendence}

As a first step, we show that the~$\alpha_n$ are transcendental for $n \geq 1$, using
Niven's theorem.

\begin{thm}[Niven~\protect{\cite[Corollary~3.12]{niven}}]\label{thm:Niven}
  Let $\trig \in \{\sin,\cos\}$ and let $x \in \Q$ with~$\trig(\pi \cdot x) \in \Q$.
  Then $\trig(\pi \cdot x) \in \{0,\pm 1/2,\pm1\}$.
\end{thm}

\begin{prop}\label{prop:antrans}
  For every $n \geq 1$, the number~$\alpha_n$ is transcendental over~$\Q$.
\end{prop}
\begin{proof}
A consequence of the Gelfond-Schneider theorem~\cite[Theorem~1]{lima} says that for any real algebraic number $x$, the expression $\arccos(x)/\pi$ is either rational or transcendental. Thus $\alpha_n$ is either rational or transcendental. 
  \emph{Assume} for a contradiction that $\alpha_n$ were rational. Then,
  because $\cos(\pi/24 \cdot \alpha_n) = 1 - 2^{-n-1}$ is also rational, by
  Niven's theorem (Theorem~\ref{thm:Niven}), we obtain
  \[  1 - \frac1{2^{n+1}} = \cos \Bigl(\frac\pi{24} \cdot \alpha_n\Bigr) \in \{0,\pm 1/2, \pm 1\}.
  \]
  However, this contradicts the hypothesis that $n \geq 1$. 
  Hence, $\alpha_n$ must be trans\-cendental.
\end{proof}

\subsection{Linear independence over the algebraic numbers}

We will now refine Proposition~\ref{prop:antrans}, using Baker's theorem. 

\begin{thm}[Baker~\cite{baker}] \label{thm:baker}
  Let $\Lambda \subset \{ \ln (\alpha) \in \mathbb{C} \mid \alpha \text{
    algebraic over~$\Q$}\}$ be linearly independent over~$\Q$. Then
  $\Lambda$ is linearly independent over the field of algebraic
  numbers.
\end{thm}

\begin{prop}\label{prop:anindep}
  Let $\Prim \subset \N$ be the set of prime numbers. Then the
  sequence~$(\alpha_{p-2})_{p \in \Prim}$ is linearly independent over
  the algebraic numbers. 
\end{prop}

For the prime $p=2$ we compute that $\alpha_{p-2} = \alpha_0 = \frac{24
  \arccos(1/2)}{\pi} = 8$, which is rational.  Hence, Proposition
\ref{prop:anindep} includes a proof that $\alpha_{p-2}$ is
transcendental for every odd prime~$p$.

\begin{proof}
We will use Baker's Theorem \ref{thm:baker}. Rewriting~$\arccos$ as 
$$
\arccos(z) = -i \cdot \ln\bigl(i \cdot z + \sqrt{1-z^2}\bigr),
$$
we see that
$$
\alpha_{p-2} = \frac{24 \cdot \arccos(1-2^{-p+1})}{\pi} = \frac{-24 \cdot i }{\pi} \cdot \ln(\gamma_p),
$$
where
$$
\gamma_p := i \cdot \frac{2^{p-1}-1}{2^{p-1}} + \frac{1}{2^{p-1}} \cdot \sqrt{2^p - 1}.
$$
We will show in Claim~\ref{claim:ln gamma lin indep} that for every 
finite set~$\{ p_1, \ldots, p_k \}$ of distinct primes the family~$\{
\ln(\gamma_{p_j}) \}_{j \in \{1, \ldots, k\}}$ is linearly
independent over~$\Q$. As $\alpha_{p-2}$ is a uniform rescaling of~$\ln(\gamma_p)$,
this will imply by using Baker's Theorem that this family is 
also linearly independent over the algebraic numbers.

We will show the linear independence of~$\{
\ln(\gamma_{p_j}) \}_{j \in \{1, \ldots, k\}}$
 over~$\Q$ in several steps: 

\begin{claim} 
  Let $(m_k)_{k \in \N}$ be a sequence of pairwise coprime positive
  integers.  Then, for every~$k \in \N_{\geq 2}$, we have that
  $$
  \sqrt{m_k} \not \in \Q[i, \sqrt{m_{1}}, \ldots, \sqrt{m_{k-1}}].
  $$
\end{claim}

\begin{proof}
  This follows from a classical result of Besicovitch~\cite{besicovitch}.
\end{proof}

\begin{claim} \label{claim: 2 power p not in field}
  Let $\{ p_1, \ldots, p_k \}$ be a finite set of distinct primes. Then 
  $$
  \sqrt{2^{p_k}-1} \not \in \Q[i, \sqrt{2^{p_{1}}-1}, \sqrt{2^{p_{2}}-1}, \ldots, \sqrt{2^{p_{k-1}} - 1}]
  $$
\end{claim}
\begin{proof}
  For all primes~$p, q \in \N$ with~$p \neq q$, the
  Mersenne numbers $2^{p}-1$ and $2^q-1$ are coprime. We may conclude
  using the previous claim.
\end{proof}

\begin{claim} \label{claim:powers lin indepd}
  Let $\{ p_1, \ldots, p_{k} \}$ be a finite set of distinct primes and let~$n \in \N_{>0}$.
  Then 
  $$
  \gamma_{p_k}^n \not \in \Q[i, \sqrt{2^{p_{k-1}}-1}, \sqrt{2^{p_{k-2}}-1}, \ldots, \sqrt{2^{p_1} - 1}].
  $$
\end{claim}

\begin{proof}
  We compute that
  \begin{align*}
    \gamma_{p_k}^n
    & = \Bigl( i \cdot \frac{2^{p_k-1}-1}{2^{p_k-1}} + \frac{1}{2^{p_k-1}} \cdot \sqrt{2^{p_k} - 1} \Bigr)^n
    \\
    & = \frac{1}{2^{n(p_k-1)}} \cdot
        \sum_{j=0}^n {n \choose j} \cdot i^{n-j} \cdot (2^{p_k-1}-1)^{n-j}\cdot (2^{p_k}-1)^{\frac{j}{2}}.  
  \end{align*}
  We see that the terms contributing to~$\sqrt{2^{p_k}-1}$ are the
  terms where $j$ is odd and that there exist~$q_1,q_2 \in \Q$
  with 
  $$
  \gamma_{p_k}^n = i^n \cdot (q_1 + q_2 \cdot i \cdot \sqrt{2^{p_k}-1}).
  $$

  \emph{Assume} for a contradiction that $q_2$ were zero.
  Then $\gamma_{p_k} \in \Q \cup i \cdot \Q$ and as
  $|\gamma_{p_k}|=1$ we obtain~$\gamma_{p_k}^n \in \{\pm 1, \pm i \}$.
  In particular, $\gamma_{p_k}$ is a root of unity. Therefore, there exists
  an~$x \in \Q$ with
  $$
  \gamma_{p_k} = \cos (2 \pi \cdot x) + i \cdot \sin (2 \pi \cdot x).
  $$
  According to Niven's Theorem~\ref{thm:Niven}, by comparing with the
  definition of~$\gamma_{p_k}$, we see that $\frac{2^{p_k}-1}{2^{p_k}}
  \in \{0, \frac 12, 1 \}$. But if $p_k$ is a prime, then this never
  happens. Hence, $q_2$ is non-zero, and so $\gamma_{p_k}^n \not \in
  \Q[i, \sqrt{2^{p_1}-1}, \ldots, \sqrt{2^{p_{k-1}}-1}]$ by Claim~\ref{claim: 2 power p not in field}.
\end{proof}

\begin{claim} \label{claim:ln gamma lin indep}
  Let $\{ p_1, \ldots, p_k \}$ be a finite set of distinct primes. Then the
  corresponding family~~$\{
\ln(\gamma_{p_j}) \}_{j \in \{1, \ldots, k\}}$ is linearly independent over~$\Q$.
\end{claim}

\begin{proof}
  \emph{Assume} for a contradiction that this family were linearly
  dependent over~$\Q$, whence over~$\Z$. Thus, there 
  are integers $n_i \in \Z$, not all zero, such that
  $$
    \ln(\gamma_{p_1}^{n_1} \cdots \gamma_{p_k}^{n_k})
  = n_1 \cdot \ln(\gamma_{p_1}) + \dots + n_k \cdot \ln(\gamma_{p_k}) =  0.
  $$
  Without loss of generality we may assume that $n_k > 0$.
  Hence, 
  $$
  \gamma_{p_1}^{n_1} \cdots \gamma_{p_k}^{n_k} \in \{1 + m \cdot 2 \pi i \mid m \in \Z \}.
  $$
  The left-hand side is algebraic over~$\Q$, but the right-hand side
  is only algebraic if $m=0$. Thus, we conclude that
  $\gamma_{p_1}^{n_1} \cdots \gamma_{p_k}^{n_k} = 1$; in other words,
  $$\gamma_{p_k}^{n_k} = \gamma_{p_1}^{-n_1} \cdots \gamma_{p_{k-1}}^{-n_{k-1}}.
  $$
  Moreover, by construction, $\gamma_{p_1}^{-n_1} \cdots
  \gamma_{p_{k-1}}^{-n_{k-1}} \in \Q[i, \sqrt{2^{p_1}-1}, \ldots,
    \sqrt{2^{p_{k-1}}-1}]$. However, this contradicts
  Claim~\ref{claim:powers lin indepd}.  Thus, $\ln(\gamma_{p_1}),
  \ldots, \ln(\gamma_{p_k})$ are linearly independent over $\Q$.
\end{proof}

This finishes the proof of Proposition~\ref{prop:anindep}.
\end{proof}

\section{Solving the group-theoretic problem}\label{sec:SLSE}

As the basic building block for our constructions we pick~$\SLS$ because
its low-degree (co)homology, its second bounded cohomology, and its
quasi-mor\-phisms are already known to basically have the right structure.

\subsection{Basic properties of~$\SLS$}

We collect basic properties of~$\SLS$
needed in the sequel; further information on the (bounded) Euler
class for circle actions can be found in the literature~\cite{BFH,Ghys_circle}.

\begin{prop}[low-degree (co)homology of~$\SLS$]\label{prop:lowdegSLS}
  \hfil
  \begin{enumerate}
  \item The group~$\SLS$ is finitely presented.
  \item The group~$H_1(\SLS;\Z)$ is finite (and non-trivial).
  \item The group~$\SLS$ does \emph{not} admit any non-trivial
    quasi-morphisms.
  \item We have~$H^2_b (\SLS;\R) \cong \R$ and the bounded
    Euler class~$\eu \R \SLS _b$ is a generator.
  \item The evaluation map~$\langle \eu \Z \SLS, \args\rangle \colon
    H_2(\SLS;\Z) \longrightarrow \Z$ has finite kernel and finite cokernel.
  \end{enumerate}
\end{prop}
\begin{proof}
  \emph{Ad~1.}
  The group~$\SLS$ can be written as an amalgamated free
  product of the form
  \[ \SLS \cong \SL_2(\Z) *_{\Gamma_0(2)} \SL_2(\Z),
  \]
  where $\Gamma_0(2)$ is the subgroup of~$\SL_2(\Z)$ of those
  matrices whose lower left entry is divisible by~$2$; this 
  leads to an explicit finite presentation~\cite[p.~81]{serretrees}.

  \emph{Ad~2.}
  In particular, one obtains that $H_1(\SLS;\Z) \cong \Z/3$ is
  finite~\cite[Proposition~3.1]{ademnaffah}.
  (Moreover, applying the Mayer-Vietoris sequence to the decomposition in
  the proof of the first part allows to
  compute the cohomology~$H^*(\SLS;\Z)$~\cite{ademnaffah}.)
 
  \emph{Ad~3.}
  This is one of many examples of groups acting on the circle
  with this property~\cite[Example~5.38]{Calegari}.
  
  \emph{Ad~4.}
  This is a result of Burger and Monod: The
  inclusion~$\SLS \longrightarrow \SL_2(\R)$ induces an
  isomorphism~$H^2_{cb}(\SL_2(\R);\R) \longrightarrow
  H^2_b(\SLS;\R)$~\cite[Corollary~24]{BurgerMonod}\cite[Corollary~4]{buchermonod}.
  Moreover, $H^2_{cb}(\SL_2(\R);\R) \cong \R$, generated by the bounded
  Euler class~\cite{burgermonodsl2}.

  \emph{Ad~5.}
  We abbreviate~$\Gamma := \SLS$.  Because $\Gamma$ is finitely
  presented, $H_2(\Gamma;\Z)$ is a finitely generated Abelian
  group~\cite[II.5]{brown}. Moreover, it has been computed that
  $H_2(\Gamma;\Q) \cong \Q$~\cite[Proposition~2.2]{moss}. Hence,
  $H_2(\Gamma;\Z)$ is virtually~$\Z$ and it suffices to show that the
  evaluation~$\langle \eu \Z \Gamma, \args\rangle \colon
  H_2(\Gamma;\Z) \longrightarrow \Z$ is non-trivial.

  As the space~$Q(\Gamma)$ of 
  quasi-morphisms (modulo trivial quasi-morphisms) is trivial, the
  comparison map~$c_\Gamma \colon H_b^2(\Gamma;\R) \longrightarrow
  H^2(\Gamma;\R)$ is injective~\cite[Theorem~2.50]{Calegari}.  In particular, $\eu \R
  \Gamma = c_\Gamma (\eu \R \Gamma _b)$ is non-trivial
  in~$H^2(\Gamma;\R)$. Therefore, by the universal coefficient
  theorem, also the evaluation $\langle \eu \Z \Gamma,\args\rangle
  \colon H_2(\Gamma;\Z) \longrightarrow \Z$ associated with the
  integral Euler class~$\eu \Z \Gamma \in H^2(\Gamma;\Z)$ is
  non-trivial.
\end{proof}

\subsection{Imitating the universal central extension}

If $\Gamma$ is a perfect group, then its universal central extension~$E$
is a perfect group that satisfies~$H_2(E;\R)\cong 0$. The universal
central extension of~$\Gamma$ can be constructed as the central extension
corresponding to the cohomology class~$\varphi$ in~$H^2(\Gamma; H_2(\Gamma;\Z))$
whose evaluation map~$\langle \varphi,\args\rangle \colon H_2(\Gamma;\Z)
\longrightarrow H_2(\Gamma;\Z)$ is the identity map. Moreover, we may compute the
quasimorphisms on~$E$ from $H^2_b(\Gamma;\R)$, which in turn allows us to
compute the stable commutator length on~$E$ using Bavard's Duality
Theorem~\cite[Section 5]{heuerloeh4mfd}. The group~$\SLS$ is not
perfect, thus it does not have a universal central extension.
Instead, we will choose a central extension of~$\SLS$ that is able
to play the same role in our context.

\begin{prop}\label{prop:H_2}
  Let $\Gamma$ be a finitely presented group with
  finite~$H_1(\Gamma;\Z)$, let $A$ be a finitely generated
  Abelian group, and let $E$ be a central 
  extension group of~$\Gamma$ that corresponds to
  a class~$\varphi \in H^2(\Gamma;H)$ such that the 
  evaluation map~$\langle \varphi,\args\rangle \colon H_2(\Gamma;\Z)
  \longrightarrow A$ has finite kernel
  and finite cokernel.
  Then:
  \begin{enumerate}
  \item The group~$E$ is finitely presented.
  \item We have~$H_1(E;\R) \cong 0$ and $H_2(E;\R) \cong 0$.
  \end{enumerate}
\end{prop}
\begin{proof}
  The central extension group~$E$ fits into a short
  exact sequence of the form
  $\xymatrix{%
    1 \ar[r]
    & A \ar[r]
    & E \ar[r]
    & \Gamma \ar[r]
    & 1.}
  $

  \emph{Ad~1.}
  Because $A$ is finitely generated, the central extension group~$E$
  of~$\Gamma$ by~$A$ is also finitely presented.

  \emph{Ad~2.}
  Because the extension is central, we have the associated
  exact sequence
  \[ \makebox[0pt]{\xymatrix@=2em{%
    H_1(E;\Z) \otimes_\Z A \ar[r] 
    & H_2(E;\Z) \ar[r]
    & H_2(\Gamma;\Z) \ar[r]^-\beta
    & A \ar[r]
    & H_1(E;\Z) \ar[r]
    & H_1(\Gamma;\Z) \ar[r]
    & 0
    }}
  \]
  by Eckmann, Hilton, and Stammbach~\cite[(1.4) and Theorem~2.2]{eckmannhiltonstammbachI}, 
  where
  \begin{align*}
    \beta \colon H_2(\Gamma;\Z) & \longrightarrow A \\
    \alpha & \longmapsto \langle \varphi,\alpha\rangle.
  \end{align*}
  By assumption, $\beta$ has finite cokernel and $H_1(\Gamma;\Z)$ is finite.
  Hence, $H_1(E;\Z)$ is finite and therefore also the left-most
  term~$H_1(E;\Z) \otimes_\Z A$ is finite. As $\beta$ has
  finite kernel, this implies that $H_2(E;\Z)$ is finite.  Applying
  the universal coefficient theorem, shows that $H_2(E;\R) \cong
  H_2(E;\Z) \otimes_\Z \R \cong 0$.
\end{proof}

With these preparations, we can now give a proof of
Theorem~\ref{theorem:SLSEnew}:

\begin{proof}[Proof of Theorem~\ref{theorem:SLSEnew}]
  We only need to combine Propositions~\ref{prop:lowdegSLS}
  and~\ref{prop:H_2}. As $\widetilde\Gamma$ is finitely generated,
  $H_1(\widetilde \Gamma;\R) \cong 0$ implies that $H_1(\widetilde
  \Gamma;\Z)$ is finite.
\end{proof}

\subsection{More on almost universal extensions}

Let us mention that the same procedure as in the previous proofs also
works in other, similar, situations:

\begin{setup}\label{setup:scl}
  Let $\Gamma$ be a group with a given orientation preserving continuous
  action on~$S^1$ with the following properties:
  \begin{itemize}
  \item The group~$\Gamma$ is finitely presented.
  \item The group~$H_1(\Gamma;\Z)$ is finite.
  \item The group~$\Gamma$ does \emph{not} admit any non-trivial quasi-morphisms.
  \item We have~$H^2_b(\Gamma;\R) \cong \R$ and the bounded Euler class~$\eu \R \Gamma_b$
    is a generator.
  \end{itemize}
  In this situation, we denote the central extension group of~$\Gamma$ associated
  with the Euler class~$\eu \Z \Gamma \in H^2(\Gamma;\Z)$ by~$\widetilde \Gamma$. 
\end{setup}

We have already seen in the previous propositions that $\SLS$ fits
into this setup. Another prominent example is Thompson's group~$T$,
which is even perfect; the condition on~$H^2_b$ follows from explicit
cohomological computations~\cite[Proposition~5.6]{heuerloeh4mfd},
based on calculations by Ghys and Sergiescu~\cite{GS}.

\begin{prop}\label{prop:euler}
  Let $\Gamma$ be as in Setup~\ref{setup:scl}. Then:
  \begin{enumerate}
  \item The evaluation map~$\langle \eu \Z \Gamma, \args\rangle \colon H_2(\Gamma;\Z)
    \longrightarrow \Z$ is non-trivial.
  \item Let $H := H_2(\Gamma;\Z)$, let $ m\in \N_{>0}$ be a generator
    of~$\im \langle \eu \Z \Gamma,\args\rangle \subset \Z$ (which is
    non-zero by the first part), and let $\varepsilon := 1/m \cdot
    \langle \eu \Z \Gamma,\args\rangle \colon H \longrightarrow \Z$. Then
    there exists a $\varphi \in H^2(\Gamma;\Z)$ with
    \[ H^2(\id_\Gamma;\varepsilon) (\varphi) = \eu \Z \Gamma
    \qand
    \langle \varphi,\args\rangle = m \cdot \id_H.
    \]
  \item Let $E$ be the central extension group of~$\Gamma$ associated
    with~$\varphi$. Then there exists an epimorphism~$\psi \colon E \longrightarrow \widetilde \Gamma$
    with~$\psi|_H = \varepsilon \colon H \longrightarrow \Z$ and~$\ker \psi \subset H$.
  \end{enumerate}
\end{prop}
\begin{proof}
  \emph{Ad~1.}
  This is the same universal coefficient theorem argument as
  in the last part of (the proof of) Proposition~\ref{prop:lowdegSLS}. 

  \emph{Ad~2.}
  By the naturality of the short exact sequence in the universal
  coefficient theorem, we have the following commutative diagram
  with exact rows:
  \[ \xymatrix{%
    0 \ar[r]
    & \Ext^1_\Z\bigl(H_1(\Gamma;\Z),H\bigr) \ar[r]
      \ar[d]_{\Ext^1(\id;\varepsilon)}
    & H^2(\Gamma;H) \ar[r]^-{\varphi \mapsto \langle \varphi,\args\rangle}
      \ar[d]_{H^2(\id_\Gamma;\varepsilon)}
    & \Hom_\Z(H,H) \ar[r]
      \ar[d]^{f \mapsto \varepsilon \circ f}    
    & 0 
    \\
    0 \ar[r]
    & \Ext^1_\Z\bigl(H_1(\Gamma;\Z),\Z\bigr) \ar[r]
    & H^2(\Gamma;\Z) \ar[r]_-{\varphi \mapsto \langle \varphi,\args\rangle}
    & \Hom_\Z(H,\Z) \ar[r]
    & 0 
  }
  \]
  The left vertical arrow is an epimorphism because $\varepsilon$ is
  an epimorphism and the exactness properties of~$\Ext$ over the
  principal ideal domain~$\Z$. Moreover, the right vertical arrow
  maps~$m \cdot \id_H$ to~$m \cdot \varepsilon = \langle
  \eu\Z\Gamma,\args\rangle$. A short diagram chase therefore proves the
  existence of the desired class~$\varphi \in H^2(\Gamma;H)$ (e.g.,
  using the four~lemma~\cite[Lemma~I.3.2]{maclanehomology}). 

  \emph{Ad~3.}
  Because the extension classes are related
  via~$H^2(\id_\Gamma;\varepsilon)(\varphi) = \eu \Z \Gamma$,
  there exists a group homomorphism~$\psi \colon E \longrightarrow
  \widetilde \Gamma$ with~$\psi|_H = \varepsilon$ that induces the
  identity on~$\Gamma$:
  \[ \xymatrix{%
    1 \ar[r]
    & \Z \ar[r]
    & \widetilde \Gamma \ar[r]
    & \Gamma \ar[r]
    & 1
    \\
    1 \ar[r]
    & H \ar[r] \ar[u]^{\varepsilon}
    & E \ar[r] \ar@{-->}[u]^{\psi}
    & \Gamma \ar[r] \ar@{=}[u]
    & 1
    }
  \]
  As $\varepsilon \colon H \longrightarrow \Z$ is an epimorphism
  also $\psi \colon E \longrightarrow \widetilde \Gamma$ is an
  epimorphism. By construction, $\ker \psi \subset H$.
\end{proof}

\begin{corr}
  Let $\Gamma$ be as in Setup~\ref{setup:scl}, let $H
  := H_2(\Gamma;\Z)$, and let~$E$ be the central extension group
  of~$\Gamma$ associated with the class~$\varphi \in H^2(\Gamma;H)$ of
  Proposition~\ref{prop:euler}.  Then:
  \begin{enumerate}
  \item The group~$E$ is finitely presented and $H_2(E;\R) \cong 0$.
  \item The epimorphism~$\psi \colon E \longrightarrow
    \widetilde \Gamma$ of Proposition~\ref{prop:euler} induces an
    isomorphism
    \begin{align*}
      Q(\psi) \colon Q(\widetilde \Gamma) & \longrightarrow Q(E)
      \\
      [f] & \longmapsto [f \circ \psi]
    \end{align*}
    and both spaces are one-dimensional. 
    Here, $Q$ denotes the space of quasi-morphisms modulo trivial
    quasi-mor\-phisms.
  \item In particular, $\scl_E([E,E]) = \scl_{\widetilde \Gamma}([\widetilde \Gamma,
    \widetilde \Gamma])$
    as subsets of~$\R$.
  \end{enumerate}
\end{corr}
\begin{proof}
  \emph{Ad~1.}
  This follows directly from Proposition~\ref{prop:H_2}.

  \emph{Ad~2.}  We will use bounded cohomology in degree~$2$ to derive
  the statement on quasi-morphisms; 
  we consider the commutative diagram
  \[ \xymatrix{%
    0 \ar[r]
    & Q(\widetilde \Gamma) \ar[r]^-\delta
      \ar[d]_{Q(\psi)}
    & H^2_b(\widetilde \Gamma;\R) \ar[r]^-{c^2_{\widetilde \Gamma}}
      \ar[d]_{H^2_b(\psi;\R)}  
    & H^2(\widetilde \Gamma;\R)
      \ar[d]^{H^2(\psi;\R)}
    \\
    0 \ar[r]
    & Q(E) \ar[r]_-{\delta}
    & H^2_b(E;\R) \ar[r]_-{c^2_{E}}
    & H^2(E;\R)
    }
  \]
  with exact rows.

  By construction, the kernel of the epimorphism~$\psi \colon E
  \longrightarrow \widetilde \Gamma$ lies in the Abelian group~$H$ and
  thus is amenable. By the mapping theorem in bounded
  cohomology~\cite[p.~40]{vbc}\cite[Theorem~4.3]{ivanov},
  the induced map~$H^2_b(\psi;\R) \colon H^2_b(\widetilde \Gamma;\R)
  \longrightarrow H^2_b(E;\R)$ is an isomorphism.

  Because~$H_2(E;\R) \cong 0$, we also have~$H^2(E;\R) \cong 0$. Therefore,
  $\delta \colon Q(E) \longrightarrow H^2_b(E;\R)$ is an isomorphism.

  We now show that also $\delta \colon Q(\widetilde \Gamma)
  \longrightarrow H^2_b(\widetilde \Gamma;\R)$ is an isomorphism: By
  the mapping theorem in bounded cohomology, the extension
  projection~$\widetilde \pi \colon \widetilde \Gamma \longrightarrow
  \Gamma$ induces an isomorphism~$H^2_b(\widetilde \pi ;\R) \colon
  H^2_b(\Gamma;\R) \longrightarrow H^2_b(\widetilde \Gamma;\R)$. As
  $H^2_b(\Gamma;\R)$ is generated by the bounded Euler class, also
  $H^2_b(\widetilde \Gamma;\R)$ is one-dimensional and generated by
  \[ \widetilde \eurm := H^2_b(\widetilde \pi;\R)(\eu \R \Gamma_b).
  \]
  By naturality of the comparison map, we obtain that
  \[ c^2_{\widetilde \Gamma} (\widetilde \eurm)
     = H^2(\widetilde \pi;\R) (\eu \R \Gamma).
  \]
  By construction of the central Euler class extension~$\widetilde
  \Gamma$, we have~$H^2(\widetilde\pi;\Z)(\eu \Z \Gamma) = 0 \in H^2(\widetilde
  \Gamma;\Z)$.  Therefore, $H^2(\widetilde\pi;\R) (\eu\R\Gamma) = 0$ and
  so~$c^2_{\widetilde \Gamma}(\widetilde \eurm) = 0$.  This shows that
  $\delta \colon Q(\widetilde \Gamma) \longrightarrow H^2_b(\widetilde
  \Gamma;\R)$ is an isomorphism.

  Now commutativity of the left square in the diagram above shows
  that $Q(\psi) \colon Q(\widetilde \Gamma) \longrightarrow Q(E)$
  is an isomorphism.

  \emph{Ad~3.}  Let $[f] \in Q(\widetilde \Gamma) \cong \R$ be a
  homogeneous generator, which exists by the second part; then $[f \circ \psi]$
  is a homogeneous generator of~$Q(E)$. Bavard
  duality~\cite{Bavard}\cite[Theorem~2.70]{Calegari} implies that
  for all~$g \in [E,E]$, we have
  \[ \scl_E (g)
  = \frac{\bigl| f \circ \psi(g)\bigr|}{2 \cdot D_E(f \circ \psi)}
  = \frac{\bigl| f(\psi(g))\bigr|}{2 \cdot D_{\widetilde \Gamma}(f)}
  = \scl_{\widetilde \Gamma}\bigl(\psi(g)\bigr);
  \]
  the defects in the denominators are equal because $\psi$ is an
  epimorphism. Again, because $\psi$ is an epimorphism, we conclude
  that $\scl_E$ and $\scl_{\widetilde \Gamma}$ have the same image in~$\R$.
\end{proof}

\section{Right-computability of simplicial volumes}\label{sec:simvolrc}

We now turn to right-computability of the numbers occuring
as simplicial volumes. After recalling basic terminology in
Section~\ref{subsec:rc}, we will prove Theorem~\ref{theorem:simvolrc}
in Section~\ref{subsec:proofsimvolrc}.

\subsection{Right-computability}\label{subsec:rc}

We use the following version of (right-)computability of real numbers,
which is formulated in terms of Dedekind cuts. For basic notions of
(recursive) enumerability, we refer to the book of Cutland~\cite{Cutland}. 

\begin{defn}[right-computable]
  A real number~$\alpha$ is \emph{right-computable} if the set~$\{x
  \in \Q \mid \alpha < x \}$ is recursively enumerable. 
  We say that $\alpha$ is \emph{computable} if both $\{x \in \Q \mid
  \alpha < x \}$ and $\{x \in \Q \mid \alpha > x \}$ are recursively
  enumerable.
\end{defn}

Further information on different notions of one-sided computability of
real numbers can be found in the work of Zheng and
Rettinger~\cite{zhengrettinger}.

There are only countably many recursively enumerable subsets of $\Q$
and thus the set of right computable and computable numbers is
countable.  

We collect some easy properties:
\begin{lemma}
  \hfil
  \begin{enumerate}
  \item
    If $\alpha, \beta \in \R_{\geq 0}$ are right-computable and
    non-negative, then so is $\alpha \cdot \beta \in \R$.
  \item
    If $\alpha \in \R_{\geq 0}$ is a real number and $c \in \R_{>0}$ a computable number such that
    $c \cdot \alpha$ is right-computable, then $\alpha$ is right-computable.
  \end{enumerate}
\end{lemma}

\begin{proof}
  For the first part we observe that if $\alpha, \beta \geq 0$, then 
  $\{ x \in \Q \mid \alpha < x \}
  \cdot \{ y \in \Q \mid \beta < y \} = \{ z \in \Q \mid \alpha \cdot
  \beta < z \}$.

  For the second part, let $\alpha \in \R_{\geq 0}$ be such that $c\cdot \alpha$ is
  right-computable, where $c$ is computable. Since $c$ is computable
  and positive, so is $c^{-1}$, thus $c^{-1}$ is in particular
  right-computable. Hence $\alpha = c^{-1} \cdot (c \cdot \alpha)$ is the
  product of non-negative right-computable numbers and thus right-computable.
\end{proof}

To a subset~$A \in \N$ we associate the number~$x_A := \sum_{n \in \N}
2^{-n}$. We relate the (right-)computability of~$x_A$ to the
computability of~$A$ as a subset of~$\N$, following Specker
\cite{specker}.
\begin{prop} \label{prop:not right computable}
Let $A \subset \N$ and let $x_A$ be defined as above. Then:
\begin{enumerate}
\item \label{item:left comp} If the set~$A$ is recursively enumerable, then $x_A$ is left-computable and $2-x_A = x_{\N \setminus A}$ is right-computable.
\item \label{item:comp} The set~$A$ is recursive if and only if $x_A$ is computable.
\item \label{item: not right comp number} If $A$ is recursively enumerable but not recursive, then $x_A$ is not right-computable.
\end{enumerate}
\end{prop}

\begin{proof}
The first two items are classical results of Specker~\cite{specker}. To see item~\ref{item: not right comp number}, let $A$ be recursively enumerable but not recursive. \emph{Assume} that $x_A$ is right-computable. By item~\ref{item:left comp}, $x_A$ is then also left-computable. Thus, $x_A$ is both left- and right-computable, whence computable. But by item~\ref{item:comp} this implies that $A$ is recursive, which contradicts our assumption.
\end{proof}

\begin{lemma}\label{lem:fractions}
  Let $f \colon \N \longrightarrow \N$ be a function with the
  following property: The set~$\{ (m,n) \in \N \times \N \mid f(m) \leq
  n\} \subset \N \times \N$ is recursively enumerable.  Then
  \[ \inf_{m \in \N_{>0}} \frac{f(m)}m
  \]
  is right-computable.
\end{lemma}
\begin{proof}
Set~$S :=\{ (m,n) \in \N \times \N \mid f(m) \leq 
  n\}$ and observe that
  $$
  \inf_{m \in \N_{>0}} \frac{f(m)}{m} = \inf_{(m,n) \in S} \frac{n}{m}.
  $$
  There is a Turing machine that, as input, takes a rational number
  and then enumerates all rational numbers above it.  We may
  diagonally use this Turing machine and the enumeration of~$S$ to
  enumerate the set
  $$
  \Bigl\{ x \in \Q \Bigm| \exi{(m,n) \in S} \frac{n}{m} < x \Bigr\}
  = \Bigl\{ x \in \Q \Bigm| \inf_{m \in \N_{>0}} \frac{f(m)}m < x \Bigr\}.
  $$
  Thus indeed $\inf_{m \in \N_{>0}} \frac{f(m)}m$ is right-computable.
\end{proof}

\subsection{Proof of Theorem~\ref{theorem:simvolrc}}\label{subsec:proofsimvolrc}

Let $M$ be an oriented closed connected manifold and $d := \dim
M$. Then $M$ is homotopy equivalent to a finite (simplicial)
complex~$T$~\cite{siebenmann,kirbysiebenmann}; let
$f \colon M \longrightarrow |T|$ be such a homotopy equivalence and for
a commutative ring~$R$ with unit, let
\[ [T]_R := H_d(f;R) \bigl([M]_R\bigr) \in H_d\bigl(|T|;R\bigr).
\]
If $R$ is a normed ring, then we write $\|\cdot\|_{1,R}$ for the
associated $\ell^1$-semi-norm on~$H_d(|T|;R)$. 
Because $f$ is a homotopy equivalence, we have
\[ \| M\| = \bigl\| [M]_\R\bigr\|_{1,\R} = \bigl\| [T]_{\R} \bigr\|_{1,\R}.
\]
Moreover, the $\ell^1$-semi-norm with $\R$-coefficients can be
computed via rational coefficients~\cite[Lemma~2.9]{mschmidt}:
\[ \| M\| = \bigl\| [T]_{\R} \bigr\|_1
          = \bigl\| [T]_{\Q} \bigr\|_{1,\Q}
          = \inf_{m \in \N_{>0}} \frac{\bigl\| m \cdot [T]_{\Z}\|_{1,\Z}}{m}.
\]
The function~$m \longmapsto \| m \cdot [T]_{\Z}\|_{1,\Z}$ satisfies
the hypothesis of Lemma~\ref{lem:fractions}
(see Lemma~\ref{lem:isvenum} below). Applying Lemma~\ref{lem:fractions}
therefore shows that the number~$\|M\|$ is right-computable. 

\begin{lemma}\label{lem:isvenum}
  In this situation, the
  subset
  \[ \bigl\{ (m,n) \in \N \times \N \bigm| \| m \cdot [T]_{\Z} \|_{1,\Z} \leq n \bigr\} \subset \N \times \N
  \]
  is recursively enumerable.
\end{lemma}
\begin{proof}
  We can use a straightforward enumeration of combinatorial
  models of cycles~\cite[proof of Corollary~5.1]{loehodd}: 

  First, $H_d(|T|;\Z)$ is isomorphic to the simplicial
  homology~$H_d(T;\Z)$ of~$T$. Therefore, we can (algorithmically)
  determine a simplicial cycle~$z$ on~$T$ that represents the
  class~$[T]_\Z$; this cycle can also be viewed as a singular cycle
  on~$|T|$.
  
  Inductive simplicial approximation of singular simplices shows that
  for every singular cycle~$c \in C_d(|T|;\Z)$, there exists a
  singular cycle~$c' \in C_d(|T|;\Z)$ with the following properties:
  \begin{itemize}
  \item The cycles $c$ and $c'$ represent the same homology class
    in~$H_d(|T|;\Z)$.
  \item The chain~$c'$ is a \emph{combinatorial singular chain}, i.e.,
    all singular simplices in~$c'$ are simplicial maps from an
    iterated barycentric subdivision of~$\Delta^d$
    to an iterated barycentric subdivision of~$T$.

    Here, each singular simplex in~$c'$ is the simplicial approximation
    of a singular simplex in~$c$. In particular, in general, the image
    of a singular simplex in~$c'$ might touch several simplices of~$T$
    and might pass them several times.
  \item We have~$|c'|_1 \leq |c|_1$.
  \end{itemize}
  This allows us to restrict attention to such combinatorial singular
  chains.  Moreover, the following operations can be performed by
  Turing machines:
  \begin{itemize}
  \item Enumerate all iterated barycentric subdivisions of~$T$
    and~$\Delta^d$.
  \item Enumerate all simplicial maps between two finite simplicial
    complexes.
  \item Hence: Enumerate all combinatorial singular $\Z$-chains of~$T$.
  \item Check, for given~$m \in \N$, whether a combinatorial singular
    $\Z$-chain on~$T$ is a a cycle and represents the class~$m \cdot
    [T]_\Z$ in~$H_d(|T|;\Z)$ (through comparison with the
    corresponding iterated barycentric subdivision of~$z$ in
    simplicial homology).
  \item Compute the $1$-norm of a combinatorial singular $\Z$-chain.
  \end{itemize}
  In summary, we can enumerate the set~$\{(m,c) \mid m \in \N, c \in
  C(m)\}$, where $C(m)$ is the set of all combinatorial $\Z$-cycles
  of~$T$ that represent~$m \cdot [T]_\Z$ in~$H_d(|T|;\Z)$.

  We now consider the following algorithm: Given~$m, n \in \N$,
  we search for elements of $1$-norm at most~$n$ in~$C(m)$.
  \begin{itemize}
  \item If such an element is found (in finitely many steps), then the
    algorithm terminates and declares that~$\|m \cdot [T]_\Z\|_{1,\Z}
    \leq n$.
  \item Otherwise the algorithm does not terminate.
  \end{itemize}

  From the previous discussion, it is clear that
  this algorithm witnesses that the set~$\{(m,n) \in \N \times \N \mid
  \|m \cdot [T]_\Z\|_{1,\Z} \leq n\}$ is recursively enumerable.
\end{proof}

This completes the proof of Theorem~\ref{theorem:simvolrc}.

\begin{rmk}
  It should be noted that the argument above is constructive enough to
  also give a slightly stronger statement (similar to the case of
  integral simplicial volume~\cite[Remark~5.2]{loehodd}): The function
  from the set of (finite) simplicial complexes (with vertices
  in~$\N$) that triangulate oriented closed connected manifolds to the
  set of subsets of~$\Q$ given by
  \[ T \longmapsto \| \, |T|\,\| 
  \]
  is semi-computable (and not only the resulting individual real
  numbers) in the following sense: There is a Turing machine that
  given such a triangulation~$T$ and $x \in \Q$ as input
  \begin{itemize}
  \item halts if $\| \,|T|\, \| < x$ and declares that $\| \,| T|\, \| < x$,
  \item and does not terminate if~$ \| \,|T|\, \| \geq x$.
  \end{itemize}
  But it is known that this function is \emph{not}
  computable~\cite[Theorem~2, p.~88]{weinberger}.
\end{rmk}

{\small
\bibliographystyle{alpha}
\bibliography{bib_l1}}

\vfill

\noindent
\emph{Nicolaus Heuer}\\[.5em]
  {\small
  \begin{tabular}{@{\qquad}l}
DPMMS,    University of Cambridge \\
    \textsf{nh441@cam.ac.uk},
    \textsf{https://www.dpmms.cam.ac.uk/$\sim$nh441}
  \end{tabular}}

\medskip

\noindent
\emph{Clara L\"oh}\\[.5em]
  {\small
  \begin{tabular}{@{\qquad}l}
    Fakult\"at f\"ur Mathematik,
    Universit\"at Regensburg,
    93040 Regensburg\\
    \textsf{clara.loeh@mathematik.uni-r.de}, 
    \textsf{http://www.mathematik.uni-r.de/loeh}
  \end{tabular}}

\end{document}